\documentclass[reqno,12pt]{amsart}

\usepackage{amsfonts,amsmath,amsthm,amssymb,amscd}
\usepackage[all]{xy}
\usepackage[mathscr]{eucal}
\usepackage{mathrsfs}

\newcommand{\version}{Ver.~0.0}
\newcommand{\setversion}[1]{\renewcommand{\version}{Ver.~{#1}}}
\setversion{0.0 [2016/04/27 16:30:15]}
\setversion{0.01 [2016/06/26 10:58:24]}
\setversion{0.02 [2016/09/15 15:26:53]}
\setversion{0.52[2016/11/28]}
\setversion{0.6 [2016/12/14 23:20:06]}
\setversion{0.9 [2017/03/17 23:02:00]}
\setversion{1.0 [2017/03/24 14:41:54]}
\setversion{1.01 [2017/03/25 11:20:26]}

\title [Enhanced adjoint orbits]
{Enhanced adjoint action and their orbits for the general linear group}

\dedicatory{}

\author{Kyo Nishiyama}
\address{
Department of Physics and Mathematics\\
Aoyama Gakuin University\\
Fuchinobe 5-10-1, Sagamihara 252-5258, Japan}
\email{kyo@gem.aoyama.ac.jp}

\thanks{K. N.  is supported by JSPS Grant-in-Aid for Scientific Research (C) \#{16K05070}.}

\author{Takuya Ohta}
\address{
  Depertment of Mathematics, Tokyo Denki University,
  Senju-asahi-cho 5, Adachi-ku,Tokyo 120-8551, Japan}
\email{ohta@cck.dendai.ac.jp}

\thanks{}

\date{\version\quad(compiled on \today)}

\newlength{\tempwidth}

\theoremstyle{plain}
\newtheorem{theorem}{Theorem}

\newtheorem{corollary}[theorem]{Corollary}
\newtheorem{lemma}[theorem]{Lemma}

\theoremstyle{definition}

\newtheorem{problem}[theorem]{\upshape Problem}

\theoremstyle{remark}
\newtheorem{remark}[theorem]{\upshape Remark}

\numberwithin{equation}{section}
\numberwithin{theorem}{section}

\newcommand{\Z}{\mathbb{Z}}

\newcommand{\R}{\mathbb{R}}

\newcommand{\C}{\mathbb{C}}

\newcommand{\lie}[1]{\mathfrak{#1}}

\newcounter{thmenum}
\newenvironment{thmenumerate}{%
\begin{list}{$(\thethmenum)$}{%
\usecounter{thmenum}
\setlength{\labelsep}{.5em}
\setlength{\labelwidth}{-7pt}
\setlength{\topsep}{0pt}
\setlength{\partopsep}{0pt}
\setlength{\parsep}{0pt}
\setlength{\leftmargin}{3pt}
\setlength{\rightmargin}{0pt}
\setlength{\itemindent}{\leftmargin}
\setlength{\itemsep}{0pt}
}}
{\end{list}}

\newcommand{\mycomment}[1]{} 

\newcommand{\vectwo}[2]{{\renewcommand{\arraystretch}{.85}\Bigl(\begin{array}{@{\,}c@{\,}}{#1}\\ {#2}\end{array}\Bigr)}}

\newcommand{\eb}{\boldsymbol{e}}

\newlength{\lengthcup}
\newcommand{\diag}{\qopname\relax o{diag}}

\newcommand{\trace}{\qopname\relax o{trace}}

\newcommand{\Stab}{\qopname\relax o{Stab}}

\newcommand{\Ad}{\qopname\relax o{Ad}}

\newcommand{\codim}{\qopname\relax o{codim}}

\renewcommand{\Im}{\qopname\relax o{Im}}

\newcommand{\Image}{\qopname\relax o{Im}}

\newcommand{\closure}[1]{\overline{#1}}

\newcommand{\transpose}[1]{\,{}^t{#1}\,}

\newcommand{\Spec}{\mathop\mathrm{Spec}\nolimits{}}

\newcommand{\Orbit}{\mathbb{O}}

\newcommand{\GITquotient}{/\!/}
\newcommand{\git}{\GITquotient}

\newcommand{\GL}{\mathrm{GL}}
\newcommand{\SL}{\mathrm{SL}}

\newcommand{\Mat}{\mathrm{M}}

\newcommand{\Det}{\qopname\relax o{Det}}

\newcommand{\skipover}[1]{}

\newcommand{\nilpotents}{\mathfrak{N}}
\newcommand{\nullcone}{\mathcal{N}}
\newcommand{\Cbatsu}{\C^{\times}}
\newcommand{\maxX}{\mathcal{X}}

\begin{document}

\begin{abstract} 
  We studied an enhanced adjoint action of the general linear group on
  a product of its Lie algebra and a vector space
  consisting of several copies of defining representations and its duals.
  We determined regular semisimple orbits (i.e., closed orbits of maximal dimension)
  and the structure of enhanced null cone, including its irreducible components and their dimensions.
\end{abstract}

\maketitle



\section*{Introduction}

Let $ G $ be a reductive algebraic group over the complex number field $ \C $, 
and $ \lie{g} $ its Lie algebra.  
The adjoint action of $ G $ on $ \lie{g} $ is a basic tool for many aspects of representation theory, and also useful for 
invariant theory, theory of singularities, and so on.  

In \cite{Achar.Henderson.2008}, \cite{Achar.Henderson.2011}, \cite{Johnson.2010}, Achar-Henderson and Johnson 
considered an enhanced version of 
nilpotent varieties and classified the nilpotent orbits (there are only finitely many of them).  
Also, in \cite{Kato.2009}, Kato considered an ``exotic''  nilpotent cone and give the 
Deligne-Langlands theory for those exotic nilpotent orbits.  
There are many related works based on algebraic geometry, combinatorial theory, and theory of character sheaves 
(\cite{Travkin.2009}, \cite{FGT.2009}, \cite{Henderson.Trapa.2012}, \cite{Fresse.Nishiyama.2016}, 
\cite{Rosso.2012}).

In these papers, enhancement of the nilpotent cone is only ``one-sided'' to get a criterion of 
finiteness of orbits.  
However, from the view point of symmetric spaces, 
it seems better to enhance all the adjoint orbits in two-sided directions.  
In this respect, we have already had two results (\cite{Ohta.2008}, \cite{Nishiyama.2014}), 
which relates the orbit structure of two enhanced actions.  But, so far, we have not known 
the explicit orbit structures of individual enhanced adjoint actions.  

In this paper, 
we begin to study (two-sided) ``enhanced adjoint action'' of $ G $ for $ G = \GL_n(\C) $ (type A).  
The big difference from those one-sided enhanced (or exotic) ones is that 
there appear infinitely many nilpotent orbits.  
So the analysis becomes more difficult, but involving less combinatorics.  
In easiest cases, we can describe enhanced adjoint orbits fairly explicitly, 
but in general, we have obtained coarser structures, like
regular orbits of maximal possible dimensions, 
structure of invariants,
irreducible components of nilpotent variety.

To state the main results more explicitly, let us introduce some notations.  
Let $ V = \C^n $ be a vector space of dimension $ n $.  
We consider a natural action of 
$ G = \GL(V) = \GL_n(\C) $ 
on 
\begin{equation*}
W = (\C^n)^{\oplus p} \oplus (\C^{\ast n})^{\oplus q} \oplus \Mat_n 
= \Mat_{n,p} \oplus \Mat_{q, n} \oplus \Mat_n ,
\end{equation*}
with the action of $ g \in G $ given by 
\begin{equation*} 
g \cdot (B, C, A) 
= (g B, C g^{-1}, \Ad(g) A) 
\qquad
\text{ for } \;\;
(B, C, A) \in \Mat_{n,p} \oplus \Mat_{q, n} \oplus \Mat_n.  
\end{equation*}
Thus, the part $ \Mat_n $ is considered to be $ \lie{g} = \lie{gl}_n(\C) $ and
the action is the adjoint action.
For the other parts, $ \Mat_{n,p} $ is a $ p $-copy of natural representations
and $ \Mat_{q, n} $ is a $ q $-copy of its dual.
So the space $ W $ is the full enhanced adjoint representation we explained.

There are obvious invariants for the action.  
We put 
\begin{align*}
\tau_k &:= \trace A^k  & & (1 \leq k \leq n - 1), \\
\gamma_{i, j}^{\ell} &:= (C A^{\ell} B)_{i, j} & & 
(0 \leq \ell \leq n - 1, \, 1 \leq i \leq q, \, 1 \leq j \leq p).
\end{align*}
These invariants are generators of the whole invariant ring $ \C[W]^G $, 
which seems to be known for experts in various forms including quiver theory 
(see Theorem~\ref{thm:generators-of-invariants}).
Thus, we can define a quotient map
$ \pi_W : W \to \C^n \times (\Mat_{q, p})^n $ using these invariants
(see Eq.~\eqref{eq:quotient-map-Phi}).

If $ p = 1 $ or $ q = 1 $, the quotient map has a very good property.
Namely, we get 

\begin{theorem}[Theorem~\ref{theorem:general-fiber-EAO} (2)]
  If $ p = 1 $ or $ q = 1 $, the map $ \pi_W : W \to \C^n \times (\Mat_{q, p})^n $
is an affine categorical quotient map (note that $ \Mat_{q, p} = \C^p $ or $ \C^q $).  
In particular, the quotient map $ \pi_W $ is coregular, and 
$ \C[W]^G $ is a polynomial ring generated by the fundamental invariants listed above.
\end{theorem}

For general $ p \geq 1 $ and $ q \geq 1 $,
the following theorem gives a generic structure of enhanced adjoint orbits.

\begin{theorem}[Theorem~\ref{theorem:general-fiber-EAO} and Corollary~\ref{corollary:reg-ss-orbits}] 
The dimension of the image 
$ \dim \Im \pi_W $ is equal to $ n (p + q) $, and 
a general fiber of $ \pi_W $ is a single $ G $-orbit of dimension $ n^2 $.  
This implies that general orbits for the enhanced adjoint action are closed of dimension $ n^2 $.  
\end{theorem}

These orbits are called \emph{regular semisimple orbits}.
Another extreme cases are nilpotent orbits.
We investigate the null cone $ \nilpotents(W) \subset W $ in \S~\ref{section:sturucture-null-cone},
and get the following results.

\begin{theorem}[Theorem~\ref{theorem:irred-decomp-null-cone}]
The null cone $ \nilpotents(W) $ is reducible and it has $ n + 1 $ irreducible components 
$ C_k \subset \nilpotents(W) \; (0 \leq k \leq n) $ given in 
Lemma~\ref{lemma:irreducible-component-Ck}.
The dimension of the null cone is $ n^2 - n + n \cdot \max \{ p, q \} $
and $ \nilpotents(W) $ is equi-dimensional if and only if $ p = q $.
\end{theorem}

Finally, we get the structure of general (enhanced) nilpotent orbits contained in each component 
$ C_k $ in
Theorem~\ref{theorem:enhanced-nilpotent-orbit-generic}.

\medskip

\noindent
\textbf{Acknowledgement.}  
We thank Minoru Itoh for useful discussions on invariants.

\section{Setting}

Let $ V = \C^n $ be a vector space of dimension $ n $.  
We consider a natural action of 
$ G = \GL(V) $ 
on 
\begin{equation*}
W = W(p, q; r) 
:= V^{\oplus p} \oplus (V^{\ast})^{\oplus q} \oplus (V \otimes V^{\ast})^{\oplus r} 
\end{equation*}
in the obvious manner.  
In explicit matrix form, we can identify 
\begin{equation*}
W = (\C^n)^{\oplus p} \oplus (\C^{\ast n})^{\oplus q} \oplus (\Mat_n)^{\oplus r} 
= \Mat_{n,p} \oplus \Mat_{q, n} \oplus \Mat_n^r ,
\end{equation*}
with the action of $ g \in G $ on 
\begin{equation*}
(B, C, (A_1, \dots, A_r)) \in \Mat_{n,p} \oplus \Mat_{q, n} \oplus \Mat_n^r 
\end{equation*}
given by
\begin{equation*} 
g \cdot (B, C, (A_1, \dots, A_r)) 
= (g B, C g^{-1}, (\Ad(g) A_i)_{i = 1}^r ) .  
\end{equation*}
There are obvious invariants, which we list below.  
For a multi-index $ I = (i_1, i_2, \dots, i_{\ell}) \; (1 \leq i_k \leq r) $, 
let us write 
$ A_I = A_{i_1} A_{i_2} \cdots A_{i_{\ell}} $.  
We denote $ [n] = \{ 1, 2, \dots, n \} $ as usual, then the multi-index $ I $ above is 
an element in $ [r]^{\ell} $.  
We put 
\begin{align*}
\tau_I &:= \trace (A_I)  & & (I \in [r]^{\ell}), \\
\gamma_{i, j}^K &:= (C A_K B)_{i, j} & & 
(K \in [r]^{\ell}, 1 \leq i \leq q, 1 \leq j \leq p),
\end{align*}
where we allow $ \ell = 0 $ for $ K $, which means $ A_K = 1_n $ (identity matrix).  
These invariants are generators of the whole invariant ring, 
which is essentially due to a more general result of Le Bruyn and Procesi 
\cite[\S~3, Theorem~1]{Bruyn.Procesi.1990}
(see also \cite{Bruyn.Procesi.1987}, \cite{Itoh.2013}).

\begin{theorem}\label{thm:generators-of-invariants}
\begin{equation*}
\C[W]^G 
= \langle \tau_I, \gamma_{i,j}^K \mid I, K \in [r]^{\ell}, \ell \geq 0, 
i \in [q], j \in [p] \rangle/\text{\upshape alg.}
\end{equation*}
\end{theorem}

\begin{proof}
In this proof, we will largely follow the notation of \cite{Bruyn.Procesi.1990}.  
Let us denote a connected quiver by $ Q $ and by $ \alpha $ its dimension vector.  
For a representation space $ R(Q, \alpha) $ of $ Q $, Theorem~1 in \cite{Bruyn.Procesi.1990} states that 
the invariant ring $ \C[Q, \alpha]^{\GL(\alpha)} $ 
is generated by traces of oriented cycles.  
So we will consider a quiver $ Q $ of two vertices $ Q_0 = \{ 1, 2 \} $ with arrows 
\begin{equation*}
 Q_1 = \{ a_i \mid 1 \leq i \leq r \} 
\cup \{ b_i \mid 1 \leq i \leq p \} 
\cup \{ c_i \mid 1 \leq i \leq q \} ,
\end{equation*}
where $ a_i $'s are loops connecting $ 1 $ and itself (i.e., $ h(a_i) = t(a_i) = 1 $); 
$ b_i $'s are arrows from $ 2 $ to $ 1 $ ($ h(b_i) = 2 $, $ t(b_i) = 1 $); 
and $ c_i $'s are arrows from $ 1 $ to $ 2 $ ($ h(c_i) = 1 $, $ t(c_i) = 2 $)．
Take a dimension vector 
$ \alpha = (\alpha(1), \alpha(2)) = (n, 1) $, so that 
$ V(1) = \C^n $ and $ V(2) = \C $．
Then our $ W = W(p, q, r) $ coincides with 
the representation space $ R(Q, \alpha) $.  

The invariants are considered with respect to the action of 
$ G(\alpha) = \GL_n \times \GL_1 $.  
However, the representation image of $ G(\alpha) $ on 
$ W = R(Q, \alpha) $ and that of $ \GL_n $ are the same 
because the action of the torus $ \GL_1 $ on $ V(2) = \C $ can be recaptured by 
the center of $ \GL_n $.  
So the both invariant rings for $ G(\alpha) $ and $ \GL_n $ are the same.  

\skipover{
However, the action of the torus $ \GL_1 $ on $ V(2) = \C $ by conjugation coincides with 
the action of the center of $ \GL_n $.   
Thus the invariants of $ G(\alpha) $ are just those for the action of $ \GL_n $, 
which coincides with ours．
}

Let us consider any closed cycles.  
Since we take traces, we can start from any vertices contained in the cycle.  
If it only contains the vertex $ 1 $, the traces are $ \tau_I $'s．
If it contains the vertex $ 2 $, 
we will start from $ 2 $ which necessarily ends in $ 2 $.   
Decompose the cycle into several cycles 
which start from $ 2 $ and end in $ 2 $.  
Since $ V(2) = \C $ is $ 1 $-dimensional, 
a decomposed cycle starting from $ 2 $ represents a scalar being equal to its trace.  
Thus the trace of the cycle which we are considering is a product of various $ \gamma_{i, j}^K $'s.  
\end{proof}

Let us denote 
$ \pi = \pi_W : W \to W \git G $, 
an affine quotient map by the action above.  
As a set, the quotient $ W \git G $ corresponds to the set of closed $ G $-orbits in $ W $.  
It is known that these closed orbits 
are precisely the set of equivalence classes of completely reducible representations of a quiver corresponding to $ W $.  

The followings are our subjects studied in this article.

\begin{problem}
\begin{thmenumerate}
\item
Let $ \nilpotents(W) = \pi_W^{-1}( \pi_W(0) ) $ be the nilpotent variety, 
which consists of the nilpotent elements $ x $ with the property $ \closure{ G \cdot x } \ni 0 $．
Investigate detailed structures of $ \nilpotents(W) $.  
In particular, we are interested in $ \dim \nilpotents(W) $;  
irreducible components;  
orbit structure; if it is reduced or not.  

\item
Are generic fibers of $ \pi $ a single orbit?  
If this is affirmative, we know generic orbits are closed.  
What are their dimensions?

\item
In general, determine the orbit space structure of $ W $.  

\item
Analyze the singularities of the quotient space 
$ W \git G $.
\end{thmenumerate}
\end{problem}

Some comments are in order.  

The nilpotent variety $ \nilpotents(W) $ is the ``worst'' fiber.  
So we are strongly interested in its structure.  
At the same time, we are also interested in the coinvariants (or the harmonics), which are 
the ``functions'' on  the nilpotent variety.  
To study it, the structure of irreducible components of $ \nilpotents(W) $ or its reducedness is very important.

On the other hand, general fibers are supposed to have ``best'' properties we can expect.  
So this will be helpful to study the quotient space, at least its smooth part.  
It would be too ambitious to expect getting very explicit orbit structure of the whole space $ W $.  
Also it seems to be a difficult problem to clarify the structure of the singularities of the quotient space.

\section{Enhanced adjoint action}

In the following, we restrict ourselves to the case 
$ r = 1 $, so that 
$ W = \Mat_{n,p} \oplus \Mat_{q, n} \oplus \Mat_n $ 
on which $ G = \GL_n $ acts.  
In the matrix form, $ g \in \GL_n $ acts on 
$ (B, C, A) \in \Mat_{n,p} \oplus \Mat_{q, n} \oplus \Mat_n $ via 
$ g \cdot (B, C, A) = (g B, C g^{-1}, \Ad(g) A) $.  
We call this action the \emph{enhanced adjoint action}.  

Now Theorem~\ref{thm:generators-of-invariants} gives a set of generators 
of $ G $-invariants:
\begin{align}
\tau_k &:= \trace (A^k)  & & (1 \leq k \leq n), 
\label{eq:tau-k}
\\
\gamma_{i, j}^k &:= (C A^k B)_{i, j} & & 
(0 \leq k \leq n - 1, \; 1 \leq i \leq q, \; 1 \leq j \leq p).
\label{eq:gamma-ij-k}
\end{align}
Note that $ A^n $ is a linear combination of $ A^k $'s $ (0 \leq k \leq n - 1) $ 
thanks to Cayley-Hamilton's formula, 
so that we don't need higher powers of $ A $ in $ \tau_k $ or $ \gamma_{i, j}^k $.  
Let us denote the affine quotient map by 
\settowidth{\tempwidth}{$ = \bigl( (\tau_k)_{k = 1}^n ; ( C A^k B)_{k = 1}^n \bigr) $}
\begin{equation}\label{eq:quotient-map-Phi}
\vcenter{\xymatrix @R-8ex @M+1.5ex @C-3ex @L+.5ex @H+1ex {
\makebox[.5ex][r]{$\pi_W : \;\;$} W \ar[rr] & & \C^n \oplus (\Mat_{q,p})^n 
& \hspace*{.8\tempwidth}
\\
(A, B, C) \ar@{|->}[rr] & & \Bigl( (\tau_k)_{k = 1}^n ; \bigl( (\gamma_{i,j}^k)_{i, j} \bigr)_{k = 0}^{n - 1} \Bigr)
\makebox[.5ex][l]{$ = \bigl( (\tau_k)_{k = 1}^n ; ( C A^k B)_{k = 0}^{n - 1} \bigr) $}
&
}}
\end{equation}
By the general theory of quotients, 
we know the image $ \Im \pi_W $ is a closed subvariety of 
$ \C^n \oplus (\Mat_{q, p})^n $.  
Let us denote by $ \Det_r(\Mat_{q, p}) $ the determinantal variety 
consisting of matrices in $ \Mat_{q, p} $ of rank less than or equal to $ r $.  
Clearly 
$ \Im \pi_W $ is contained in 
$ \C^n \times \Det_n(\Mat_{q, p})^n $.  

\begin{theorem}\label{theorem:general-fiber-EAO}
Under the setting above, the image $ \Im \pi_W $ is isomorphic to 
the affine quotient $ W \git G = \Spec(\C[W]^G) $.  We have: 
\begin{thmenumerate}
\item\label{theorem:general-fiber-EAO:item:dim-quotient-space}
There is a dominant map 
\begin{equation*}
\Psi : \C^n \times (\Det_1(\Mat_{q, p}))^n \to \Im \pi_W ,
\end{equation*}
whose restriction to a dense open subset of  
$ \C^n \times (\Det_1(\Mat_{q, p}))^n $ gives an affine quotient map 
under the natural action of $ S_n $ to a dense open subset of $ \Im \pi_W $.  
Consequently, we get 
$ \dim W \git G = \dim \Im \pi_W = n (p + q)  $, and 
a general fiber of $ \pi_W $ is of dimension $ n^2 $.  
\item\label{theorem:general-fiber-EAO:item:p=1-or-q=1}
If $ p = 1 $ or $ q = 1 $, the quotient map $ \pi_W $ is surjective, and 
$ \Im \pi_W = \C^n \oplus (\Mat_{q, p})^{\oplus n} $ is an affine space.  
In particular, the quotient map $ \pi_W $ is coregular, and 
$ \C[W]^G $ is a polynomial ring of the fundamental invariants listed in 
\eqref{eq:tau-k} and \eqref{eq:gamma-ij-k}.
\end{thmenumerate}
\end{theorem}

\begin{proof}
Let us fix a generic diagonal matrix 
$ A = t = \diag(t_1, \dots, t_n) $, where 
$ t_i \neq t_j \; (i \neq j) $.  
For $ 1 \leq r \leq n $, put 
\begin{equation*}
X^{(r)} = \begin{pmatrix}
c_{1,r} \\ c_{2, r} \\ \vdots \\ c_{p, r} 
\end{pmatrix}
( b_{r, 1}, b_{r, 2}, \dots, b_{r,q}) \in \Det_1(\Mat_{q, p}), 
\end{equation*}
where $ c_{i, j} $ denotes the $ (i, j) $-element of the matrix $ C \in \Mat_{q, n} $ and 
similarly $ b_{i,j} $ for $ B \in \Mat_{n, p} $.
We get 
\begin{equation}\label{eq:definition-of-Gamma}
C A^k B = (\gamma_{i,j}^k)_{i, j} 
= \Bigl( \sum_{r = 1}^n c_{i,r} t_r^k b_{r, j} \Bigr)_{i, j} 
= \sum_{r = 1}^n t_r^k X^{(r)} =: \Gamma^{(k)} .
\end{equation}
Thus, in the matrix form, it holds 
\begin{align}
\label{eq:DX=Gamma}
&
\begin{pmatrix}
1 & 1 & \cdots & 1 \\
t_1 & t_2 & \cdots & t_n \\
\vdots & \vdots & \ddots & \vdots \\
t_1^{n -1} & t_2^{n -1} & \cdots & t_n^{n -1} 
\end{pmatrix}
\begin{pmatrix}
X^{(1)} \\
X^{(2)} \\
\vdots \\
X^{(n)} 
\end{pmatrix}
=
\begin{pmatrix}
\Gamma^{(0)} \\
\Gamma^{(1)} \\
\vdots \\
\Gamma^{(n-1)} 
\end{pmatrix}, 
\intertext{hence}
&
\begin{pmatrix}
X^{(1)} \\
X^{(2)} \\
\vdots \\
X^{(n)} 
\end{pmatrix}
=
D(t)^{-1} 
\begin{pmatrix}
\Gamma^{(0)} \\
\Gamma^{(1)} \\
\vdots \\
\Gamma^{(n-1)} 
\end{pmatrix}, 
\label{eq:X-is-D-inverse-Gamma}
\end{align}
where $ D(t) = (t_j^{i - 1})_{i, j} $ 
denotes the Vandermonde matrix in the former equation \eqref{eq:DX=Gamma}.
We define a map $ \Psi $ 
from 
$ U = \C^n \times (\Det_1(\Mat_{q, p}))^n $ to $ \Im \pi_W $ by
\footnote{The map $ \Psi $ is a priori defined to 
be one from $ U $ to $ \C^n \oplus (\Mat_{q, p})^n $.  
However, since $ \Image \pi_W $ is closed, we know the image of $ \Psi $ is contained in $ \Im \pi_W $. 
See below.}
\begin{equation}\label{eq:Sn-invariant-map-Psi}
\vcenter{
\xymatrix @R-8ex @M+1.5ex @C-3ex @L+.5ex @H+1ex {
\makebox[.1ex][r]{$\Psi : \;\;$} 
U \;\; \ar@{->}[rr] & & \Im \pi_W
& \hspace*{.8\tempwidth}
\\
(t; (X^{(k)})_{k = 1}^n) \ar@{|->}[rr] & & 
\Bigl( (\sum\limits_{i = 1}^n t_i^k )_{k = 1}^n ; (\Gamma^{(k)})_{k = 1}^n 
\makebox[.5ex][l]{$ = D(t) (X^{(k)})_{k = 1}^n \Bigr) $}
&
}}
\end{equation}
The map $ \Psi $ is generically an $ n! $-fold covering map, 
and it is invariant under $ S_n $ which acts on $ U $ by the diagonal coordinate permutation on 
the both factor
\footnote{Unfortunately, $ \Psi $ may not be a quotient map.  
See Remark~\ref{remark:Sn-quotient-map}}.

Take $ (\tau; (\Gamma(k))_k) \in \Im \pi_W $ 
for which $ \tau $ is in an image of regular semisimple $ A $.  
Those elements consist an open dense set $ (\Im \pi_W)' \subset \Im \pi_W $.  
Then we can recover 
$ X(k) $'s via the formula 
\eqref{eq:X-is-D-inverse-Gamma}, 
if we pick a $ t $ from the fiber of $ \tau $ so that $ t_i \neq t_j \; (i \neq j) $ hold.
Thus we have a surjective covering map from an open dense subset of $ U $ to 
$ (\Im \pi_W)' $.  
Consequently, the image $ \Im \Psi $ is contained in $ \Im \pi_W $.
Thus we conclude $ \Psi : U \to \Im \pi_W $ is dominant, and 
\begin{equation*}
\dim \Im \pi_W = \dim U = n + n ( p + q - 1) = n (p + q), 
\end{equation*}
where we used 
$ \dim \Det_1(\Mat_{q, p}) = p + q - 1 $.  
Comparing the dimension, we know the dimension of a generic fiber 
of $ \pi_W $ is $ n^2 = \dim W - \dim \Im \pi_W $.  

Now let us assume $ p = 1 $ or $ q = 1 $.  
Then $ \Mat_{q, p} = \C^q $ or $ \C^p $, and
get $ \dim (\C^n \oplus (\Mat_{q, p})^{\oplus n}) = n ( p + q) = \dim \Im \pi_W $
(the last equality follows from \eqref{theorem:general-fiber-EAO:item:dim-quotient-space}).
Since the image $ \Im \pi_W $ is closed in $ \C^n \oplus (\Mat_{q, p})^{\oplus n} $, 
we have a surjective quotient map $ \pi_W : W \to \C^n \oplus (\Mat_{q, p})^{\oplus n} $ 
so that $ W \git G \simeq \C^n \oplus (\Mat_{q, p})^{\oplus n} $, an affine space.  
This means the invariants are algebraically independent and 
$ \C[W]^G $ is a polynomial ring.
\end{proof}

\begin{corollary}\label{corollary:reg-ss-orbits}
Let us denote the quotient map by 
$ \pi_W : W \to \C^n \oplus (\Mat_{q, p})^n $ as in 
\eqref{eq:quotient-map-Phi}.  
And assume that 
$ (\tau; \Gamma) = (\tau; (\Gamma^{(k)})_{k = 1}^n) \in \C^n \oplus (\Mat_{q, p})^n $ 
satisfies the following conditions {\upshape (i)} and {\upshape (ii)}.
\begin{itemize}
\item[(i)\;]\ 
There exists a regular diagonal matrix $ t $ with 
$ \tau = (\tau_k(t))_{k = 1}^n $, 
i.e.,   
$ \tau \in \C^n $ with the $ k $-th coordinate being 
$ \tau_k = \sum_{i = 1}^n t_i^k $, where $ t_i \neq t_j \; (i \neq j) $.
\item[(ii)]\ 
$ \Gamma^{(k)} \; (0 \leq k \leq n-1) $ 
corresponds to $ X^{(k)} $ via \eqref{eq:X-is-D-inverse-Gamma}, 
which are of rank $ 1 $.
\end{itemize}
Then 
$ (\tau; \Gamma) $ is in the image $ \Im \pi_W $ and 
the $ \dim \pi_W^{-1}(\tau; \Gamma) = n^2 $, i.e., 
the fiber of $ (\tau; \Gamma) $ 
is generic and of dimension $ n^2 $.  
Moreover, it is a single closed $ G $-orbit.
\end{corollary}

\begin{proof}
By condition (i), 
we can choose a regular diagonal matrix $ t $ with 
$ \tau = (\tau_k(t))_{k = 1}^n $.  
Thus we can define $ (X^{(k)}) = D(t)^{-1} \Gamma $ via \eqref{eq:X-is-D-inverse-Gamma}.  
If $ X^{(k)} $ is of rank $ 1 $, 
then we can write $ X^{(k)} = c_k \transpose{b_k} $ for certain 
$ c_k \in \C^q, \, b_k \in \C^p $.  
From these vectors, we can restore $ \transpose{B} = (b_1, \dots, b_n) $ and $ C = (c_1, \dots, c_n) $.  
Thus $ (\tau; \Gamma) = \pi_W(t, B, C) \in \Im \pi_W $.  

There is not so much choice for the fiber.  
We know the fiber over $ \tau $ of the adjoint quotient is just the conjugation of $ t $, 
which is of dimension $ n^2 - n $.  
For $ B $ and $ C $, since any column of $ B $ and $ C $ is nonzero, 
we can only multiply scalars column by column, which is of dimension $ n $.  

It is now clear that any element in the fiber can be obtained from $ (t, B, C) $ 
through the action of $ G $. 
Since the stabilizer of the fiber $ (t, B, C) $  is trivial, we again get the right dimension $ n^2 $.  
\end{proof}

\begin{remark}
Let us assume $ p = 1 $ or $ q = 1 $.  
In this case, the action of $ G = \GL_n(\C) $ on $ W $ is coregular, i.e., 
the quotient space is an affine space and the generators listed in 
\eqref{eq:tau-k} and 
\eqref{eq:gamma-ij-k}
are algebraically independent. 

However, 
if we consider an action of the simple group $ \SL_n(\C) $ instead of $ \GL_n(\C) $, 
this action is not coregular (coregular actions are classified for simple groups, see 
\cite{Schwarz.coregular.1978}, \cite{Adamovich.Golovina.coregular.1983}).

To see this, let us assume $ p = q = 1 $ for simplicity.  
Consider 
two invariants $ D_1, D_2 $ with respect to the action of $ \SL_n $ defined as follows.
For $ (u, v, A) \in V \oplus V^{\ast} \oplus \Mat_n $ (we consider $ V = \C^n $ as a column vector), 
we put 
\begin{equation*}
D_1(u, v, A) = \det \begin{pmatrix} v \\ v A \\ v A^2 \\ \vdots \\ v A^{n - 1} \end{pmatrix}, \qquad
D_2(u, v, A) = \det (u, A u, A^2 u, \dots, A^{n - 1} u)
\end{equation*}
Both $ D_1 $ and $ D_2 $ are clearly $ \SL_n $-invariants, and 
they are not $ \GL_n $-invariants so that 
they cannot be expressible by using $ \tau_k $ and $ \gamma^k $ above \footnote{
Note that, since $ p = q = 1 $, we do not need subscription $ i $ and $ j $ for $ \gamma_{i, j}^k $}.  
However, it is easy to see 
\begin{equation*}
D_1 \cdot D_2 = \det \bigl( v A^{i + j} u \bigr)_{i,j} = \det (\gamma^{i + j})_{i, j}
\end{equation*}
which gives a relation.  
This shows that the action of $ \SL_n $ is not coregular.  

When $ p > 1 $ or $ q > 1 $, similar arguments lead to the same conclusion.

However, even if it is not coregular, it seems the $ \SL_n $-orbit structure has good properties.  
We will discuss it in future.
\end{remark}

\begin{remark}\label{remark:Sn-quotient-map}
Let us consider a toy model for the map \eqref{eq:Sn-invariant-map-Psi}.  
Assume that $ V $ is a vector space and 
$ S_n $ acts on $ \C^n \times V^n $ as the diagonal coordinate permutation.  
\begin{equation*}
\xymatrix @R-1ex @M+1.5ex @C-3ex @L+.5ex @H+1ex {
\C^n \times V^n \ni (a_1, \dots, a_n; v_1, \dots, v_n) 
\ar[d]^{\psi}
\ar[rrd]^{\pi}
& & 
\\
\C^n \times V^n \ni \Bigl( (\sum_{i = 1}^n a_i^k)_{k = 1}^n ; ( \sum_{i = 1}^n a_i^k v_i)_{k = 1}^n  \Bigr) 
& & \ar[ll]_-{\varphi} (\C^n \times V^n) / S_n 
}
\end{equation*}
Consider a closed set 
$ Z = \{ (a; v) \mid a_i v_i = u  \; (1 \leq i \leq n) \} $ 
for a fixed non-zero vector $ u  $, which is stable under the $ S_n $-action.  
The image $ \psi(Z) $ does not contain an element of the form $ (0 ; w) $, 
however its closure contains $ (0; (n \, u , 0, \dots, 0)) $.  
Thus the image $ \psi(Z) $ is not closed, hence $ \psi $ is not a quotient map.
\end{remark}


\section{Structure of the null cone}\label{section:sturucture-null-cone}

We will study the structure of null cone 
$ \nilpotents(W) = \pi_W^{-1}(\pi_W(0)) $ in this section.  
For this, we follow the strategy of Popov \cite{Popov.2003} and 
Kraft and Wallach \cite{Kraft.Wallach.2006}.  
We briefly recall their theory.

\subsection{}\label{subsection:theory-Kraft-Wallach-Popov}

In this subsection, we consider a general situation so that 
the notation is independent of those in the former (sub)sections.  

Let $ G $ be a connected reductive algebraic group $ G $ over $ \C $, which acts on a vector space $ V $ linearly.  
Let $ \pi : V \to V \git G $ be the quotient map, and 
\begin{equation*}
\nullcone_V := \pi^{-1}(\pi(0)) = \{ v \in V \mid \closure{G v} \ni 0 \} 
\end{equation*}
the null cone.
For any one parameter subgroup (abbreviated as ``1-PSG'') 
$ \lambda : \Cbatsu \to G $, 
we define 
$ V(\lambda) := \{ v \in V \mid \lim_{t \to 0} \lambda(t) v = 0 \} $.  
Then $ v \in V $ is in the null cone $ \nullcone_V $ if and only if 
$ v \in V(\lambda) $ for a suitable 1-PSG $ \lambda $ (Hilbert-Mumford criterion).  

Let $ T \subset G $ be a maximal torus.  We fix $ T $ once and for all, 
and denote by $ X^{\ast}(T) $ the character group of $ T $.  
Then $ V $ has the weight space decomposition 
\begin{equation*}
V = \bigoplus_{\gamma \in X^{\ast}(T)} V_{\gamma}, 
\qquad
V_{\gamma} := \{ v \in V \mid t v = \gamma(t) v \;\; (t \in T) \}.
\end{equation*}
We denote the set of 1-PSGs $ \lambda : \Cbatsu \to T $ by $ X_{\ast}(T) $.  
Then there is a natural pairing $ \langle - , - \rangle : X_{\ast}(T) \times X^{\ast}(T) \to \Z $ determined as follows.    
For $ (\lambda, \gamma) \in X_{\ast}(T) \times X^{\ast}(T) $, 
$ m = \langle \lambda, \gamma \rangle $ if $ \gamma(\lambda(t)) = t^m \; (t \in \Cbatsu) $.  

With these notations, for a 1-PSG $ \lambda : \Cbatsu \to T \subset G $, 
we have 
\begin{equation*}
V(\lambda) = \bigoplus_{\langle \lambda, \gamma \rangle > 0} V_{\gamma} .
\end{equation*}
Since every 1-PSG of $ G $ is conjugate to a certain $ \lambda \in X_{\ast}(T) $, we get 
\begin{equation*}
\nullcone_V = \bigcup_{\lambda \in X_{\ast}(T)} G \cdot V(\lambda) .
\end{equation*}
In this decomposition, 
there appear only finitely many different 
$ V(\lambda) \neq 0 $.  
Thus, a maximal $ V(\lambda) $ may contribute to an irreducible components of $ \nullcone_V $ 
(but not always).  
We call such $ U = V(\lambda) $ a maximal unstable subspace, and put  
$ \maxX_U := \{ \gamma \in X^{\ast}(T) \mid V_{\gamma} \subset U \} = \{ \gamma \mid \langle \lambda, \gamma \rangle > 0 \} 
$, 
a maximal unstable subset of weights.  
Let $ \maxX_1, \dots, \maxX_s $ 
be a complete set of representatives of 
maximal unstable subsets of weights up to the conjugation of 
the Weyl group $ W_G(T) $, 
and $ U_i = \bigoplus_{\gamma \in \maxX_i} V_{\gamma} \; (1 \leq i \leq s) $ 
the corresponding maximal unstable subspace.  

For a 1-PSG $ \lambda $, put 
\begin{equation*}
P(\lambda) := \{ g \in G \mid \text{ the limit }\; \lim_{t \to 0} \Ad (\lambda(t)) \, g \; \text{ exists} \}.
\end{equation*}
Then $ P(\lambda) $ is a parabolic subgroup which leaves $ V(\lambda) $ stable (Kempf\cite{Kempf.1978}). 
If $ U = V(\lambda) $ is a maximal unstable subspace, 
then the stabilizer $ \Stab_G(U) $ contains $ P(\lambda) $ and hence it is a parabolic subgroup.  

Define 
$ P_i := \Stab_G(U_i) $ for each $ 1 \leq i \leq s $.
Thus, we get a natural multiplication map 
$ G \times_{P_i} U_i \to C_i \subset \nullcone_V $, where 
$ C_i = G \cdot U_i $.  
Since $ G/P_i $ is projective, the image $ C_i $ is closed and irreducible.  
Thus we can choose $ C_1, \dots, C_r $ which give irreducible components of $ \nullcone_V $ 
after renumbering if necessary.  
In this way, we can determine the irreducible decomposition of $ \nullcone_V $:
\begin{equation}\label{eq:irred-decomp-NV-general}
\nullcone_V = \bigcup_{k = 1}^r C_k .  
\end{equation}
\indent
Let us apply this theory to our situation of the enhanced adjoint representation.

\subsection{}

\newcommand{\ee}{\varepsilon}

Now let us return back to our original notation, so that 
$ G = \GL_n(\C) $ which acts on 
$ W = \Mat_{n,p} \oplus \Mat_{q, n} \oplus \Mat_n $ 
as before.  
It is easy to see that the set of weights of $ W $ is given by 
\begin{equation*}
\Lambda = \Lambda(W) := 
\{ 0 \} \cup \Delta_n \cup \{ \pm \ee_i \mid 1 \leq i \leq n \}, 
\quad
\Delta_n = \{ \ee_i - \ee_j \mid 1 \leq i \neq j \leq n \} .
\end{equation*}
Here, 
$ \Delta_n $ denotes the set of roots of type $ A_{n - 1} $ and 
$ \ee_i $ denotes the standard basis in $ \lie{t}^{\ast} $, where 
$ \lie{t} $ is the Lie algebra of the diagonal torus $ T \subset G $.  
The multiplicity of $ \alpha \in \Delta_n $ is one, while 
the multiplicity of $ \alpha = 0 $ is $ n $; that of $ \ee_i $ is $ p $ and 
that of $ -\ee_i $ is $ q $. 
We describe a family of maximal unstable subsets of weights up to the Weyl group conjugation.  
Take a standard positive system $ \Delta_n^+ = \{ \ee_i - \ee_j \mid 1 \leq i < j \leq n \} $ of $ \Delta_n $.  

\begin{lemma}\label{lemma:max-unstable-subset-of-weights}
For $ 0 \leq k \leq n $, put 
$ X_k := \Delta_n^+ \cup \{ \ee_i \mid 1 \leq i \leq k \} \cup \{ - \ee_j \mid k < j \leq n \} $.  
Then $ X_0, X_1, \dots, X_n $ gives a complete system of representatives of 
maximal unstable subset of weights up to the conjugation of the Weyl group $ W_G(T) = S_n $.
\end{lemma}

\begin{proof}
Let $ X $ be a maximal unstable subset corresponding to a 1-PSG $ \lambda $.  
Taking conjugation of $ \lambda $ by $ S_n $, 
we can assume $ \lambda = ( \lambda_1, \dots, \lambda_n) , \; 
\lambda_1 > \lambda_2 > \cdots > \lambda_n $.  
Note that, if an equality appears among $ \lambda_i $'s, the corresponding unstable subset is not maximal.  
If $ \lambda_k > 0 \geq \lambda_{k + 1} $, 
$ X $ is given by $ X_k $.
\end{proof}

Let $ U_k \subset W $ be the maximal unstable subspace corresponding to $ X_k $ so that 
\begin{equation}\label{eq:definition-Uk}
\begin{aligned}
U_k = \bigoplus_{\alpha \in X_k} W_{\alpha} 
&= \{ (\xi, \eta, v) \in \Mat_{n,p} \oplus \Mat_{q, n} \oplus \Mat_n \mid 
\\
& \qquad \qquad
\xi_{i,j} = 0 \; ( i > k ), \eta_{i',\, j'} = 0 \; ( j' \leq k ), v \in \lie{n}^+
\} ,
\end{aligned}
\end{equation}
where $ \lie{n}^+ $ denotes a maximal nilpotent subalgebra consisting of upper triangular matrices 
with $ 0 $'s on the diagonal.  
It is the Lie algebra of the unipotent radical of a Borel subgroup 
$ B $ of upper triangular matrices in $ G = \GL_n $.  
Note that 
\begin{equation*}
\xi = \vectwo{\xi_1}{0} \quad 
(\xi_1 \in \Mat_{k, p}) , \qquad \text{ while } \qquad 
\eta = (0, \eta_2) \quad 
(\eta_2 \in \Mat_{q, n - k}) .
\end{equation*}

\begin{lemma}\label{lemma:irreducible-component-Ck}
Let $ U_k \; (0 \leq k \leq n) $ be a maximal unstable subspace as above.  
Then the stabilizer $ P_k = \Stab_G(U_k) $ of $ U_k $ is the Borel subgroup $ B $ for any $ k $ 
and $ \psi_k : G \times_B U_k \to C_k \subset \nilpotents(W) $ is a resolution of singularity.  
In particular, $ C_k $ is an irreducible closed subvariety in $ \nilpotents(W) $ of 
dimension $ (n^2 - n) + p k + q (n - k) $.  
\end{lemma}

\begin{proof}
Since $ P_k $ stabilizes $ \lie{n}^+ $, it is contained in $ B $.  
On the other hand, clearly $ B $ stabilizes $ U_k $, hence $ P_k = B $.  

Let us show generic fiber of the map $ \psi_k $ is a one point set.  
Since $ C_k \supset U_k $, 
we will examine the fiber of $ (\xi, \eta, v) \in U_k $, where 
$ v \in \lie{n}^+ $ is a principal nilpotent element.  
Take an element $ [g, (\xi', \eta', u)] \in \psi_k^{-1}((\xi, \eta, v)) $.  
Then $ (\xi, \eta, v) = \psi_k([g, (\xi', \eta', u)]) = (g \xi', \eta' g^{-1}, \Ad(g) u) $.  
In particular, we have 
$ v = \Ad(g) u \in \Ad(g) \lie{b} =: \lie{b}^g $.  
It is well known that a principal element belongs to a unique Borel subalgebra.  
Since $ v \in \lie{b} $, we conclude 
$ \lie{b} = \lie{b}^g $, hence $ g \in B $.  
%
Now we know 
$ [g, (\xi', \eta', u)] \sim [1_n, (\xi, \eta, v)] $, 
which means the element in the fiber is uniquely determined.

The set of elements 
$ \{ (\xi, \eta, v) \in C_k \mid \text{$ v $ is principal nilpotent} \} $ 
is open dense in $ C_k $, 
so the map $ \psi_k $ is generically one-to-one,
hence it is birational.  
Since $ G \times_N U_k $ is a vector bundle over a projective variety,
the map $ \psi_k $ is proper and it is a resolution.
\end{proof}

\begin{theorem}\label{theorem:irred-decomp-null-cone}
Let $ \nilpotents(W) $ be the null cone, and denote $ C_k \subset \nilpotents(W) \; (0 \leq k \leq n) $ as in 
Lemma~\ref{lemma:irreducible-component-Ck}.
\begin{thmenumerate}
\item\label{theorem:irred-decomp-null-cone:item:1}
$ \nilpotents(W) = \cup_{k = 0}^n C_k $ gives the irreducible decomposition.  
So the null cone has $ (n + 1) $ components, the number of which is independent of $ p \geq 1 $ and $ q \geq 1 $.
The dimension of $ \nilpotents(W) $ is $ n^2 - n + n \cdot \max \{ p, q \} $.  
\item\label{theorem:irred-decomp-null-cone:item:2:equi-dim}
The null cone $ \nilpotents(W) $ is equidimensional if and only if $ p = q $.  
In this case, the dimension of $ \nilpotents(W) $ is $ n^2 - n + p n $.
\item\label{theorem:irred-decomp-null-cone:item:3:equi-dim-fiber}
The dimension of $\nilpotents(W) $ is $ n^2 $ if and only if $ p = q = 1 $.  
If this is the case, 
any fiber $ \pi_W^{-1}((\tau; \Gamma)) $ of $ (\tau; \Gamma) \in \Im \pi_W $ is of dimension $ n^2 $.
\end{thmenumerate}
\end{theorem}

\begin{proof}
From Lemma~\ref{lemma:irreducible-component-Ck}, the subvariety $ C_k $ is closed and irreducible.  
The general theory described in \S~\ref{subsection:theory-Kraft-Wallach-Popov} gives 
the irreducible decomposition of $ \nilpotents(W) $ 
(cf.~Equation \eqref{eq:irred-decomp-NV-general}).  
Since 
$ \dim C_k = (n^2 - n) + p k + q (n - k) $, we have 
\begin{equation*}
\dim \nilpotents(W) = \max_{0 \leq k \leq n} \Bigl\{ (n^2 - n) + p k + q (n - k) \Bigr\} 
= n^2 - n + n \cdot \max \{ p, q \}.
\end{equation*}
This proves 
\eqref{theorem:irred-decomp-null-cone:item:1}.  
The claim \eqref{theorem:irred-decomp-null-cone:item:2:equi-dim} follows immediately from 
\eqref{theorem:irred-decomp-null-cone:item:1}.  

Let us prove \eqref{theorem:irred-decomp-null-cone:item:3:equi-dim-fiber}.  
For any $ (\tau; \Gamma) \in \Im \pi_W $, 
the dimension of the fiber $ \pi_W^{-1}((\tau; \Gamma)) $ is greater than or equal to that of a general fiber, which is $ n^2 $ 
by Theorem~\ref{theorem:general-fiber-EAO}.  
On the other hand, the dimension of the null cone is the greatest among those of the fibers (see \cite{Popov.Vinberg.1994}).  
This completes the proof.
\end{proof}

\subsection{Orbits in the null cone}

Let us investigate orbits in an irreducible component 
$ C_k = G \cdot U_k \subset \nilpotents(W) $ (cf.~\eqref{eq:definition-Uk}).  
So pick $ w = ( \xi, \eta, v) \in U_k $, where $ v \in \lie{n}^+ $ is a principal nilpotent element.  
We denote the $ G $ orbit through $ w $ by $ \Orbit(w) $.  

We compute the stabilizer $ Z_G(w) $ of $ w $.  
Up to $ G $ conjugacy, we can assume 
\begin{equation*}
v = e := 
\begin{pmatrix}
 0 & 1 & \\
 & 0 & 1 & \\
 & & \ddots & \ddots & \\
 & &        & 0      & 1 \\
 & &        &        & 0   
\end{pmatrix} .
\end{equation*}
By direct calculation, we get 
\begin{equation}\label{eq:stabilizer-of-e-g}
Z_G(e) = \exp \left( \bigl\{ \sum_{i = 0}^{n - 1} \sigma_i e^i \mid \sigma_i \in \R \bigr\} \right) 
\ni \sum_{j = 1}^n x_j e^{j - 1} =: g.
\end{equation}
Assume that $ k \geq n - k $, and   
we denote $ \xi \in \Mat_{n, p} $ and $ \eta \in \Mat_{q, n} $ as 
\begin{equation}\label{eq:xi-eta}
\xi = 
\begin{pmatrix}
\xi_1
\\
0
\end{pmatrix}
\quad 
(\xi_1 \in \Mat_{k, p}), 
\qquad
\eta = ( \; 0 \; | \, \eta_1 \, ) 
\quad
(\eta_1 \in \Mat_{q, n - k}).
\end{equation}
Here we take 
\begin{equation}\label{eq:xi-1}
\xi_1 = ( \eb_k, \xi_1' ) \qquad (\xi_1' \in \Mat_{k, p - 1}) ,
\end{equation}
where $ \eb_k \in \C^k $ is the $ k $-th elementary vector whose $ k $-th coordinate is $ 1 $ and the other coordinates 
are zero.  
Then, the element $ g $ in \eqref{eq:stabilizer-of-e-g} stabilizes 
$ \xi $ and $ \eta $ if and only if $ x_1 = 1, \, x_2 = \cdots = x_{k} = 0 $.  
Thus we get 
$ Z_G(w) = \{ 1_n + \sum_{j = k + 1}^n x_j e^{j - 1} \} $.  
In particular, 
we know 
$ \codim \Orbit(w) = n - k $.  
For the orbit $ \Orbit(w) $, we can take $ \xi_1' $ in \eqref{eq:xi-1} and $ \eta_1 $ in \eqref{eq:xi-eta} freely, 
and they are uniquely determined by the orbit.  
So there is a fibration of orbits $ \Orbit(w) $ with the base space 
$ \Mat_{k, p -1} \times \Mat_{q, n - k} $ of dimension 
\begin{align*}
\dim \Orbit(w) + \dim \Mat_{k, p -1} {\times} \Mat_{q, n - k}
&= n^2 - (n - k) + k (p - 1) + q (n - k) 
\\
&= n^2 - n + k p + (n -k) q = \dim C_k .
\end{align*}
This means the family of orbits $ \{ \Orbit(w) \} $ makes up an open dense subset of the irreducible component $ C_k $.  
Since the orbits of the largest possible dimension constitute an open set, 
$ \dim \Orbit(w) = n^2- n + k $ is the largest among the orbits in $ C_k $.  
For the family parametrized by $ \Mat_{k, p -1} \times \Mat_{q, n - k} $, 
there is no reason to specialize the first column of $ \xi $.  
So, if the $ k $-th row of $ \xi $ does not vanish, we can follow the same arguments.  

Notice that this construction also applies to the case of $ k \leq n - k $, considering $ \eta $ instead of $ \xi $.  

Let us summarize what we have proven here.

\begin{theorem}\label{theorem:enhanced-nilpotent-orbit-generic}
Let 
$ C_k \subset \nilpotents(W) \; (0 \leq k \leq n) $ be 
an irreducible component of the null cone $ \nilpotents(W) $ (see Lemma~\ref{lemma:irreducible-component-Ck}).
The largest dimension of the nilpotent orbits in $ C_k $ 
is $ n^2 - \min \{ k, n - k \} $.  
Moreover, there exists an open dense subset of $ C_k $ 
which is fibered over an affine space of dimension $ k p + q (n - k) - \max \{ k , n - k \} $ 
with the fiber of isomorphic nilpotent orbits $ \Orbit $ of the largest dimension.

In particular, an irreducible component $ C_k $ contains a nilpotent orbit of dimension $ n^2 $ 
if and only if $ k = 0 $ or $ n $.
\end{theorem}

\begin{remark}
Let us consider $ w =(\xi, \eta, v) \in U_k $ as above.  
Even if $ v $ is not principal, a $ G $-orbit $ \Orbit(w) $ through $ w $ can attain 
the largest possible dimension in the irreducible component $ C_k $.  
It seems rather subtle to describe when an orbit $ \Orbit(w) $ has the largest dimension.
\end{remark}

\renewcommand{\MR}[1]{}


\def\cftil#1{\ifmmode\setbox7\hbox{$\accent"5E#1$}\else
  \setbox7\hbox{\accent"5E#1}\penalty 10000\relax\fi\raise 1\ht7
  \hbox{\lower1.15ex\hbox to 1\wd7{\hss\accent"7E\hss}}\penalty 10000
  \hskip-1\wd7\penalty 10000\box7} \def\cprime{$'$} \def\cprime{$'$}
  \def\Dbar{\leavevmode\lower.6ex\hbox to 0pt{\hskip-.23ex \accent"16\hss}D}
\providecommand{\bysame}{\leavevmode\hbox to3em{\hrulefill}\thinspace}
\providecommand{\MR}{\relax\ifhmode\unskip\space\fi MR }
\providecommand{\MRhref}[2]{%
  \href{http://www.ams.org/mathscinet-getitem?mr=#1}{#2}
}
\providecommand{\href}[2]{#2}

\end{document}